\documentclass[11pt]{article}

\usepackage[margin=1in]{geometry}
\usepackage[utf8]{inputenc}
\usepackage[T1]{fontenc}
\usepackage{amsmath}
\usepackage{amsthm,enumerate} 
\usepackage{amssymb}
\usepackage{listings}
\usepackage{url}
\usepackage{hyperref}
\usepackage{footmisc}
\usepackage{booktabs}
\usepackage{soul}
\usepackage{xcolor}
\usepackage{color}
\usepackage{dsfont}
\usepackage{tikz-cd}
\usepackage{cleveref}
\usepackage{circledsteps}
\usepackage{comment}
\usepackage[normalem]{ulem}

\hypersetup{
    colorlinks,
    linkcolor={red!50!black},
    citecolor={blue!50!black},
    urlcolor={blue!80!black}
}

\usepackage{anysize}

\newtheorem{thm}{Theorem}[section]
\newtheorem{obs}[thm]{Observation}
\newtheorem{lemma}[thm]{Lemma}
\newtheorem{theorem}[thm]{Theorem}
\newtheorem{cor}[thm]{Corollary}
\newtheorem{prop}[thm]{Proposition}
\newtheorem{notation}[thm]{Notation}
\newtheorem{remark}[thm]{Remark}

\newtheorem{de}[thm]{Definition}
\newtheorem{example}[thm]{Example}

\newtheorem{problem}[thm]{Problem}

\newcommand{\Gauss}[2]{{\begin{bmatrix} #1 \\ #2 \end{bmatrix}_q}}
\newcommand{\allones}{\mathds{1}}
\newcommand{\charvec}{\mathds{1}}

\newcommand{\EKR}{\mathcal{F}}
\newcommand{\PS}{\mathrm{PS}}
\newcommand{\cF}{\mathfrak{F}}

\newcommand{\philipp}[1]{\textcolor{blue}{#1}}

\renewcommand{\le}{\leqslant}
\renewcommand{\leq}{\leqslant}
\renewcommand{\ge}{\geqslant}
\renewcommand{\geq}{\geqslant}

\newcommand{\gauss}[2]{{#1\brack #2}}

\title{The largest sets of non-opposite chambers in spherical buildings of type $B$}

\author{
  Jan De Beule\thanks{Department of Mathematics and Data Science, Vrije Universiteit Brussel, Pleinlaan 2, B-1050 Brussel, Belgium. \texttt{jan.de.beule@vub.be}} \and
  Philipp Heering\thanks{Justus-Liebig-Universität, Mathematisches Institut, Arndtstraße 2, D-35392 Gießen, Germany. \texttt{philipp.heering@math.uni-giessen.de}} \and
Sam Mattheus\thanks{Department of Mathematics and Data Science, Vrije Universiteit Brussel, Pleinlaan 2, B-1050 Brussel, Belgium. \texttt{sam.mattheus@vub.be}} \and
  Klaus Metsch\footnotemark[2]\thanks{Justus-Liebig-Universität, Mathematisches Institut, Arndtstraße 2, D-35392 Gießen, Germany.  \texttt{klaus.metsch@math.uni-giessen.de}}
}

\begin{document}

\maketitle

\begin{abstract}

The investigation into large families of non-opposite flags in finite spherical buildings has been a recent addition to a long line of research in extremal combinatorics, extending classical results in vector and polar spaces. This line of research falls under the umbrella of Erd\H{o}s-Ko-Rado (EKR) problems, but poses some extra difficulty on the algebraic level compared to aforementioned classical results. From the building theory point of view, it can be seen as a variation of the center conjecture for spherical buildings due to Tits, where we replace the convexity assumption by a maximality condition.

In previous work, general upper bounds on the size of families of non-opposite flags were obtained by applying eigenvalue and representation-theoretic techniques to the Iwahori-Hecke algebras of non-exceptional buildings. More recently, the classification of families reaching this upper bound in type $A_n$, for $n$ odd, was accomplished by Heering, Lansdown, and Metsch. For buildings of type $B$, the corresponding Iwahori-Hecke algebra is more complicated and depends non-trivially on the type and rank of the underlying polar space. Nevertheless, we are able to find a uniform method based on antidesigns and obtain classification results for chambers (i.e.\ maximal flags) in most cases. 
%all cases, except type $^2A_{4n-3}$.
\end{abstract}



\section{Introduction}

The Erd\H{o}s-Ko-Rado (EKR) theorem \cite{OriginalEKR}, proved in 1938 but published only in 1961,  is a foundational result in extremal set theory. It states that for a finite set $X$ of size $n$, the maximum intersecting family of $k$-element subsets of $X$, where $2k+1 \leq n$, is obtained when all subsets share a common element. 
Since then, numerous generalizations and extensions of the problem have been studied across countless combinatorial structures, including vector spaces \cite{ekr_vectorspaces, hsieh}, groups \cite{EKR_permutations_fixed_cycles, stanton_ekr_chevalleygroups}, and various kinds of geometries \cite{adriansen_ekr_circlegeometries, Theorems_of_EKR_type_in_geometrical_settings, metsch_chromaticnumberoftwo}. For an overview, see \cite{EllisSurvey, GodsilMeagher}.
In all of these results, the typical main questions are: \emph{Given a set with a suitable notion of "intersection", how large can a subset of intersecting objects be? What is the structure of the largest possible such set?} \\

An important tool in the study of EKR-problems is the theory of graphs. By constructing a graph $G$ whose vertices are adjacent if the corresponding objects are not intersecting, we can utilize methods from spectral graph theory to study independent sets of the graph. Specifically, the ratio bound \cite{hoffman_bound_haemers}, also known as the Hoffman-Delsarte bound, has been instrumental in determining the maximum possible size of EKR-sets in various settings. In order to effectively apply this method, one needs to have precise control over the eigenvalues of $G$. While this is difficult in general, the EKR-problems cited earlier typically have the advantage that the adjacency matrix of $G$ can be studied by algebraic means: either it is contained in a commutative matrix algebra (known as association schemes), or $G$ is a Cayley graph of a well-understood group. In both cases we have access to well-known tools in order to compute the eigenvalues of $G$. This is not the case in this work. The graph we will work with can be seen as a Schreier coset graph and its adjacency matrix is contained in a non-commutative matrix algebra. 

Not only does this complicate the application of the ratio bound, it moreover implies that known techniques such as the polyhedral, width and linear programming methods are unavailable, see the excellent monograph by Godsil and Meagher \cite{GodsilMeagher} or the introduction of recent work by Filmus and Lindzey \cite{uniqueness_algebraically} for an overview of these techniques. Instances in the literature where classification results are obtained using algebraic methods, despite this extra layer of difficulty, are rare. We only know of the result due to Filmus and Lindzey on partially 2-intersecting perfect hypermatchings \cite[Section 9]{uniqueness_algebraically}, which relied on the representation theory of the symmetric group, and the recent result due to Heering, Lansdown and Metsch for chambers in type $A_n$, $n$ odd \cite{klaus_jesse_philipp}. In fact, our approach for type $B$ is similar to theirs. However, whereas \cite{klaus_jesse_philipp} relied on lengthy combinatorial calculations and an explicit description of the generators of the eigenspace corresponding to the smallest eigenvalue (see \cite{AlgebraicApproach2}) to find antidesigns, we give a streamlined algebraic argument that does not use the description of the eigenspaces.\\ 

In this paper we will study EKR-problems in the context of buildings, specifically for flags of finite spherical buildings. Here `intersecting' typically refers to the notion of non-oppositeness from building theory. Buildings of type $A$ come from projective spaces, and EKR-problems for specific flags of projective spaces have been studied, see \cite{line-plane-flags, point-hyperplane-flags, point-plane-flags, k-spaces_intersecting_k-2, heering_PG4q, heeringmetsch2023secondmax, plane-solid-flags}. We note that the proofs of all these results, which in most cases deal with partial flags of size at most 2, are combinatorial in nature. The complexity of these combinatorial arguments quickly increases as the partial flags and the ambient space become bigger, which explains the lack of general results. 

A similar situation occurs for buildings of type $B$, whose geometric models are polar spaces. EKR-theorems for subspaces of polar spaces can be found in \cite{boeck_ekr_polarspaces, ihringer_ekr_polarspaces, EKR_polarspace_smallsubspaces,metsch_pairwise-non-opp-lines, EKR_polarspace_generators}, but we alert the reader that different notions of `intersecting' are used. As far as we are aware, all classification results for EKR-sets in type $B$ deal with single-element flags.
Even in this case, algebraic arguments comprise an essential component in most of the proofs since a purely combinatorial or geometrical approach is quickly limited by 
the complicated and numerous possible geometrical relations that arise.

Only recently this problem has been approached from a more algebraic perspective  \cite{AlgebraicApproach,AlgebraicApproach2}, by leveraging known results on the Iwahori-Hecke algebra of spherical buildings in order to apply the ratio bound. The classification problem in type $A_n$, $n$ odd, has been achieved by Heering, Lansdown and Metsch \cite{klaus_jesse_philipp} using the method of antidesigns. The underlying idea behind the proofs in this paper are similar, but very different in execution as we will indicate. The extra difficulty comes from the fact that the geometry of polar spaces is more involved than that of projective spaces, and the non-trivial dependence of the underlying Iwahori-Hecke algebras on the type and rank of the building. \\

A second motivation for this work is the connection to the center conjecture for spherical buildings due to Tits. This conjecture was solved for non-exceptional spherical buildings by M\"uhlherr and Tits \cite{MuhlherrTits06} and for exceptional spherical buildings by Leeb and Ramos-Cuevas \cite{LR11}, and Ramos-Cuevas \cite{RamosCuevas13}.
Roughly speaking, for a spherical building $\Delta$ with associated group of Lie type $G$ (possibly defined over an infinite field), this conjecture states that a convex subset $S$ of chambers in $\Delta$ must either contain a pair of opposite chambers, or $\mathrm{Stab}_G(S)$ fixes a (non-empty) partial flag of $\Delta$. This partial flag is referred to as a center.
Convexity in this setting is defined by considering the building as a simplicial complex, we refer the reader to \cite[1.3]{Titsbook} for more details. Another, more combinatorial formulation of the center in relation to EKR-problems was given by Ihringer, Metsch, and M\"uhlherr \cite{IMM18}. 

Our main results imply that for several finite buildings of type $B$, we can replace the assumption of convexity by the requirement that the family of non-opposite chambers attains the ratio bound and still find a center (in either formulation). It is an interesting question whether this implication remains true for smaller maximal families of non-opposite chambers. In other words, what is the threshold up until which we can omit the convexity assumption from the center conjecture for finite buildings? \\

\textbf{Notation and terminology.} We will state our results in the language of finite polar spaces \cite{distance-regular-graphs, Cameron}, which are the geometrical models of spherical buildings of type $B$ defined over a finite field. The finite classical polar spaces are the geometries naturally associated to non-degenerate, respectively non-singular, sesquilinear, respectively quadratic forms on a vector space over a finite field. These geometries are naturally embedded in finite projective spaces. Their {\em elements}, also called {\em subspaces}, are the subspaces of the projective space that are totally isotropic, respectively totally singular, with respect to the sesquilinear, respectively quadratic form. The subspaces of maximal dimension contained in the polar space are called {\em generators}. Throughout this article, the word {\em subspace} will always refer to a subspace of the polar space. The {\em rank of a subspace} is its vector space dimension, a subspace of rank $i$ will be called an {\em $i$-space}. Points are $1$-spaces and generators are $n$-spaces, with $n$ the Witt index of the underlying form. The {\em rank} of a polar space is the rank of its generators. Furthermore, we define the {\em type} $e = \log_q (N-1)$, where $N$ is the number of generators through an $(n-1)$-space. By $\PS(n,e,q)$ we will denote a polar space of rank $n \geq 3$, type $e \in \{0,\frac{1}{2},1,\frac{3}{2},2\}$ and defined over $\mathbb{F}_q$. 

Let $S$ be a subspace of a polar space with underlying vector space $V$, then $S^\perp$ denotes its {\em tangent space}, this is the subspace of $V$ that is orthogonal to $S$ with relation to the underlying form. If $V$ has dimension $d$ and $S$ is an $i$-space, then $S^\perp$ has rank $d-i$.

A \emph{flag} in a polar space $\PS(n,e,q)$ is a set $F$ of nontrivial subspaces such that $U\le W$ or $W\le U$ for all $U,W\in F$ and $\{\dim(U)\mid U\in F\}$ is the \emph{type} of $F$. For $J \subseteq [n] := \{1,2,\ldots,n\}$, we will denote the set of flags of type $J$ by $\cF_J$. If $J=\{s\}$ is a singleton, we simply write $\cF_s$ instead of $\cF_{\{s\}}$. When given a flag $F$ of type $J \subseteq [n]$, we will typically denote its subspace of dimension $i$ by $F_i$. When $J = [n]$, the flag is called a \emph{chamber}.

Two flags $F$ and $F'$ of type $J$ are \emph{opposite}, if $F_i^\perp \cap F'_i =\{0\}$ (i.e.\ the zero-dimensional subspace), for all $i \in J$.

\begin{de}
    An \emph{EKR-set of flags of type $J$}, $J \subseteq [n]$, in $\PS(n,e,q)$ is a set of flags of type $J$ such that no two are opposite. 
\end{de}

For example, when $J = \{1\}$, an EKR-set of points is exactly a set of pairwise collinear points. It is clear that the maximum EKR-sets of points hence consist of all points in a generator. As far as other examples of EKR-sets of flags in type $B$ go, only the cases $J=\{n\}$ and $J=\{2\}$ have received attention in the literature \cite{metsch_pairwise-non-opp-lines, EKR_polarspace_generators}. For convenience, we record the known families for $J = \{1\}$ and $J = \{n\}$: 

\begin{example}\label{E: EKR for s-spaces}
    The following are examples of EKR-sets of $s$-spaces in $\PS(n,e,q)$, $s \in \{1,n\}$:
    \begin{enumerate}[(a)]
        \item Fix a generator $\pi$. The set of all points contained in $\pi$ is a maximum EKR-set of points and all maximum EKR-sets of points are of this form.
        \item Fix a point $P$. The set of all generators containing $P$ is an EKR-set of generators.
        \item For $e = 0$ and $n$ odd, there are two equally sized classes of generators, often called greeks and latins. The set of generators of one class is a maximum EKR-set of generators. Note that this example is larger than the previous.
        \item For $e = 1$ and $n$ odd, we can find an embedded hyperbolic space of the same rank in a parabolic space (for all $q$) or a symplectic space (for even $q$), and the previous example embeds as a maximum EKR-set of generators in this case. 
    \end{enumerate}
\end{example}

For the examples involving generators, it was shown by Pepe, Vanhove and Storme \cite{EKR_polarspace_generators} that the list above exhausts all maximum EKR-sets of generators when $n \geq 4$, except when $e = 1/2$ and $n$ is odd, in which case we do not even know the correct order of magnitude of the size of a maximum EKR-set. It might be the case that example (b) is also the maximum example in this case for sufficiently large rank. We refer to \cite{EKR_polarspace_generators} for more details. 

Starting from these examples, we can `blow them up' to produce EKR-sets of flags of type $J$, whenever $1 \in J$ or $n \in J$.

\begin{obs} \label{E: blow-ups of s-spaces}
    Let $\EKR_s$ be an EKR-set of $s$-spaces in $\PS(n,e,q)$ with $s \in [n]$. For any $J \subseteq [n]$ such that $s \in J$, we define $\EKR_J$ to be the set of all flags $F$ of type $J$ with $F_s\in \EKR_s$. Then $\EKR_J$ is an EKR-set of flags of type $J$.
\end{obs}

Our first main result is that for almost all polar spaces, the maximum EKR-sets of chambers, have the form of \Cref{E: blow-ups of s-spaces} for $s\in \{1,n\}$. In other words, the maximum EKR-sets of chambers are blow-ups of maximum EKR-sets of points or generators. This answers a question due to the first, third and fourth author \cite[Problem 3.19]{AlgebraicApproach2} for the respective values of $e$ and $n$. 

\begin{theorem}\label{thm:main}
    Consider $\PS(n,e,q)$ for $n \geq 3$ and $e\geq 1$. % except the case $e = 1/2$ and $n$ odd. 
    Then for sufficiently large $q$ (in terms of $n$) a maximum EKR-set of chambers has the structure described in \Cref{E: blow-ups of s-spaces} for $s\in \{1,n\}$.
\end{theorem}

Note that for $n = 3$, there might exist exceptional EKR-sets of generators of maximum size not mentioned in \cite{EKR_polarspace_generators}, yet our result still holds. 
The cases $e=0$ and $e=1/2$ are not covered by the theorem. 
If $e = 0$ and $n$ is odd, we have an exceptional case due to the occurrence of example (c) in \Cref{E: EKR for s-spaces}. In particular, the largest EKR-sets of generators are known and behave differently.
The case $e=1/2$, particularly for odd $n$, remains poorly understood.
As such, \textbf{we will assume for the remainder of the paper that} $e \geq 1$. \\ % \textbf{or} $n$ \textbf{is even}. \\

It is natural to ask what happens when $1,n \notin J$. To answer this question, we first introduce the graph-theoretic framework from \cite{AlgebraicApproach,AlgebraicApproach2} in which we will prove our results. For this purpose, we denote for $J \subseteq [n]$, the opposition graph on $\cF_J$ by $\Gamma_{J}$. To be precise, this is the graph whose vertices are the elements of $\cF_J$, i.e.\ the flags of type $J$ in a polar space $\PS(n,e,q)$, where two vertices are adjacent if the corresponding flags are opposite. If $J=\{s\}$, we simply write $\Gamma_s$. 

It is clear that EKR-sets of flags of type $J$ are independent sets in $\Gamma_J$. In \cite{AlgebraicApproach} the spectrum of $\Gamma_J$ was investigated in order to obtain an upper bound on the size of maximum EKR-sets of flags of type $J$, using the ratio bound. In general, we call an EKR-set \emph{ratio-sharp} if it is a maximum EKR-set that attains the ratio bound. 

\begin{theorem}\cite[Theorem 3.15]{AlgebraicApproach}\label{thm:upperbound}
Let $\PS(n,e,q)$ be a polar space with $e \geq 1$. % or $n$ be even.
Furthermore, let $J\subseteq[n]$.
Then an EKR-set $\EKR$ of flags of type $J$ of $\PS(n,e,q)$ satisfies
    \[|\EKR| \leq \frac{|\cF_J|}{q^{n+e-1}+1}.\]
    Moreover, if we have equality, then the characteristic vector of $\EKR$ lies in the sum of the eigenspaces of the largest and smallest eigenvalues of $\Gamma_J$.
\end{theorem}

As a consequence of \Cref{thm:main}, we obtain a classification of maximum EKR-sets of flags of type $J$ in buildings of type $B$, whenever $1 \in J$ or $n \in J$.

\begin{cor} \label{C: classification for flags}
    Consider $\PS(n,e,q)$ for $n\geq 4$, assume that $e \geq 1$, and let $J \subseteq [n]$ so that $1 \in J$ or $n \in J$ (or both). Then for sufficiently large $q$ (in terms of $n$) a ratio-sharp EKR-set of flags of type $J$ has the structure described in 
    %(the appropriate generalization of) 
    \Cref{E: blow-ups of s-spaces} for $s\in \{1,n\}$.
\end{cor}

We are now in the position to state our second main result: when $1,n \notin J$ the bound from \Cref{thm:upperbound} can never be sharp. This answers \cite[Problem 3.20]{AlgebraicApproach}. We remark that a special case of the following theorem (\Cref{P: finish proof}) is used to prove \Cref{thm:main}.

\begin{theorem}\label{thm:main2}
    Consider $\PS(n,e,q)$ for $n\geq 4$, assume that $e\ge 1$, and let $J \subseteq [n]$ be so that $1,n \notin J$. Then for sufficiently large $q$ (in terms of $n$) the bound in \Cref{thm:upperbound} can never be attained.
\end{theorem}

Remarkably, our restrictions on $e$ and $n$ in \Cref{thm:main} come only from the fact that these are the cases in which the ratio bound is tight. In particular the cases where $e\ge 1$ include instances in which the minimal eigenvalue of $\Gamma_{[n]}$ is attained on a well-understood module of the adjacency algebra denoted by $M_{\{[n-1],[1]\}}$, and instances (namely $e=1$ and $n$ odd) in which the minimal eigenvalue of $\Gamma_{[n]}$ is attained on the module $M_{\{\emptyset,[n]\}}$. There is a `nice' description for $M_{\{[n-1],[1]\}}$, but the module $M_{\{\emptyset,[n]\}}$ is much more mysterious, for details see \cite{AlgebraicApproach2}. \\

Returning to the matter at hand, our proof strategy for \Cref{thm:main} is as follows. We leverage the extra information from \Cref{thm:upperbound} on the characteristic vector of a tentative maximum independent set in $\Gamma_{[n]}$. For $s\in [n]$ we use the smallest eigenvalue of $\Gamma_{s}$ to find multisets of vertices that are design-orthogonal to a maximum independent set $\EKR$. This means that they have constant intersection (counting with multiplicity) with $\EKR$. For convenience, we will refer to these special multisets as \textit{antidesigns}. The design-orthogonality allows us to obtain local geometric information about $\EKR$. With geometric arguments that involve weights, we can show that $\EKR$ is a blow-up of an EKR-set of $s$-spaces for some $s \in [n]$. Finally, we once again use antidesigns for $\Gamma_s$, and we show that the `blow-ups' are too small when $1<s<n$.

%\begin{center}
%\begin{tabular}{llc}
%Form & Name & $e$\\
%\hline
%quadratic & hyperbolic & $0$\\
%hermitian & hermitian & $1/2$\\
%alternating & symplectic & $1$ \\
%quadratic & parabolic & $1$\\
%hermitian & hermitian & $3/2$\\
%quadratic & elliptic & $2$
%\end{tabular}
%\end{center}

\section{Combinatorics of polar spaces}

In this preliminary section we provide some combinatorial results on polar spaces that will be used throughout the paper. As mentioned in the introduction it is understood that we only consider subspaces of $\PS(n,e,q)$. Let $S$ and $T$ be two subspaces of the underlying vector space, then we call them \emph{skew} if $S \cap T = \{0\}$. A chamber $C$ is said to \textit{contain} a fixed $s$-space $S$, if $C_s=S$, we use the same terminology for flags. For the first lemma recall that the totally isotropic subspaces contained in a totally isotropic subspace constitute a projective geometry.

\begin{lemma} \label{L: gaussian coefficient}
	The number of $m$-subspaces contained in a fixed $s$-subspace of $\PS(n,e,q)$ is 
   \begin{align*}
 \gauss{s}{m}_q=
 \begin{cases}
\displaystyle \prod_{i=1}^m\frac{q^{s-m+i}-1}{q^i-1} & \text{if $0\le m\le s$,}
\\
0 & \text{otherwise.} \end{cases}
\end{align*}
\end{lemma}

\begin{lemma}[see Lemma 9.4.1 in \cite{distance-regular-graphs} and Remark 4.1.2 in \cite{Vanhove_PhD}] \label{L: extend subspaces} 
Let $0\le m\le  s\le n$. The number of $s$-spaces through a fixed $m$-space of $\PS(n,e,q)$ is given by 
$$ \Phi_m^s(n,e,q):= \Gauss{n-m}{s-m} \prod\limits_{i=1}^{s-m}(q^{n-m+e-i}+1).$$
If $n$, $e$ and $q$ are clear from the context, we write $\Phi_m^s$ instead of $\Phi_m^s(n,e,q)$.
\end{lemma}

\begin{lemma} \label{L:z_s}
    Let $S$ be an $s$-space of $\PS(n,e,q)$. The number of flags of type $\{1,\ldots, s\}$ that contain $S$ is $z_s(q):=\prod\limits_{i=1}^s \Gauss{i}{1}$. Since $q$ is always clear from the context, we write $z_s$ instead of $z_s(q)$.
\end{lemma}

\begin{proof}
We use induction on $s$. For $s=1$ the only flag consists of $S$ itself and $z_1=1$.
%\red{for $s=1$ it is a point?}.  \philipp{Removed the case $s=0$, do you think it is ok like this?}
If $s \geq 2$, there exist $\gauss{s}{1}$ subspaces $H$ of codimension one in $S$. The induction hypothesis applied to $H$ shows that each such subspace lies in  $z_{s-1}$ flags. Hence, the total number of flags in question is $z_{s-1}\gauss{s}{1}=z_s$. 
\end{proof}

We restate a special case of \Cref{L: extend subspaces} and \Cref{thm:upperbound} using this new notation.

\begin{cor} \label{C: Phi}
\begin{enumerate} [(a)]
\item The number of $s$-spaces in $\PS(n,e,q)$ is $\Phi_0^s$.
\item For $e\ge 1$ the size of an EKR-set of chambers of $\PS(n,e,q)$ is at most $\Phi_1^n z_n$.
\end{enumerate}
\end{cor}

For a $q$-polynomial $f$, we denote by $\deg(f)$ the degree of $f$, we also use this notation if the exponents of $f$ are half-integers.

\begin{lemma} \label{L:degree_stuff}
For $1\leq s\leq n$, we have
    \begin{enumerate} [(a)]
        \item $\deg(z_s)=s(s-1)/2$,
        \item $\deg(\Phi_m^s)=(s-m)(2n-s-m+e-\frac{s-m+1}{2})$,
        \item $\deg(z_sz_{n-s}\Phi_s^n)=s(s-1)/2+(n-s)(n-s+e-1)$,
        \item $\deg(\Phi_1^nz_n)=(n-1)(n+e-1)$.
    \end{enumerate}
\end{lemma}
\begin{proof}
For $s\in [n]$ we have that $\Gauss{n}{s}$ is a $q$-polynomial (with leading coefficient $1$). It follows directly from the definition (see \Cref{L: gaussian coefficient}) that $\deg\left(\Gauss{n}{s}\right)=s(n-s)$. The results follow from elementary computations using Gauss summation. 
\end{proof}

Furthermore, we need some results that are a bit more technical.

\begin{lemma} [Lemma 9.4.2 in \cite{distance-regular-graphs}] \label{L:advanced_comb}
Let $U$ be a fixed $m$-space in $\PS(n,e,q)$. The number of $(j+k+l)$-spaces $W$ meeting $U$ in a $j$-space and $U^\perp$ in a $(j+l)$-space is 
$$ q^{l(m-j)+k(2n-m-j-2l+e-1)-k(k-1)/2}\Gauss{m}{j}\Gauss{m-j}{k}\Gauss{n-m}{l} \prod\limits_{i=0}^{l-1}(q^{n+e-m-i-1}+1). $$
\end{lemma}

\begin{lemma} \label{L: extend flags}
	Consider a flag $F$ of type $\{t_1,t_2, \ldots ,t_f \}$ in $\PS(n,e,q)$ with $t_1 < t_2 < \dots < t_f$. Then the number of chambers in $\PS(n,e,q)$ that contain $F$ is exactly
  $$\Phi_{t_f}^n z_{t_1}z_{n-t_f}\prod_{i=2}^f(z_{t_i-t_{i-1}}). $$ 
\end{lemma}
\begin{proof}
By \Cref{L: extend subspaces} a subspace of dimension $t_f$ is contained in $\Phi_{t_f}^n$ generators. So, the only question left is the following. Given a flag $F$ of type $\{t_1,t_2, \ldots ,t_f,n \}$, how many chambers exist that contain $F$. Notice that this is a combinatorial question on vector spaces. Let $T_i$ be the subspace of $F$ that has dimension $t_i$. Put $T_0:=\{0\}$ and $t_0=0$. Then $t_i$ is the dimension of $T_i$ for $i=0,\hdots,f$ and the dimension of $T_i/T_{i-1}$ is $t_i-t_{i-1}$. Using \Cref{L:z_s}, the statement follows.
\end{proof}

The following statement can be proven by induction and a double count using \Cref{L:advanced_comb}. We give an alternative proof using buildings. For all terminology, we refer to \cite{AlgebraicApproach}.

\begin{lemma} \label{L: basicnumbers}
If $C$ is a chamber of $\PS(n,e,q)$ and $S$ a subspace of dimension $s$ with $S\cap C_s^\perp =\{0\}$, then the number of chambers that contain $S$ and are opposite to $C$ is exactly $q^{\deg(z_sz_{n-s}\Phi_s^n)}$.
\end{lemma}

\begin{proof}
This follows from \cite[Corollary 3.2]{brouwer_eigenvalues}, which states that this number is exactly $q^{\ell}$, where $\ell$ is the weighted length with respect to the building parameters of the longest word in the Coxeter group $W_J \leq W(B_n)$, with $J = [n] \setminus \{s\}$. In this case $W_J = W(A_{s-1}) \times W(B_{n-s})$ and so $\ell = s(s-1)/2+(n-s)(n-s+e-1)$. See also \cite[Remark 2.19]{AlgebraicApproach}. The statement follows from \Cref{L:degree_stuff}.
\end{proof}

\begin{remark}
    Note that $z_sz_{n-s}\Phi_s^n$ is exactly the number of chambers that contain a given $s$-space. Since the leading coefficient of $z_sz_{n-s}\Phi_s^n$ is $1$, the ``majority'' of chambers through the given $s$-space is opposite to $C$ in the above lemma.
\end{remark}

In the proof of the following lemma and throughout \Cref{S:EKR_s-spaces}, we will use the notion of \textit{quotient polar space}. Recall that given a degenerate polar space of rank $n$ with radical $R$ of rank $r$, we can quotient out the radical to obtain a non-degenerate polar space of rank $n-r$ (but of the same type and defined over the same field). In this way, whenever we want to argue about the subspaces incident with $R$, it suffices to investigate the quotient space. For such a $t$-subspace $T$ incident with $R$, we will denote by $T/R$ its quotient, which is a $(t-r)$-subspace of the polar space of rank $n-r$. In particular, we will use this idea when we want to argue about the subspaces incident with a given $m$-space $M$ in a $\PS(n,e,q)$: we can consider the degenerate polar space $M^\perp \cap \PS(n,e,q)$ with radical $M$ and quotient $M$ out to obtain a polar space denoted by $M^\perp/M$. For this non-degenerate polar space, we have $M^\perp /M \cong\PS(n-m,e,q)$. 

\begin{lemma}\label{L: heavy power m-space}
    Let $S$ be an $s$-space and $M$ an $m$-space of $\PS(n,e,q)$ with $m \leq s$, such that $M\cap S^\perp=\{0\}$. Then the number of $s$-spaces containing $M$, opposite to $S$ is 
    $q^{\deg(\Phi_m^s)}$.
\end{lemma}

\begin{proof}
    From $M \cap S^\perp = \{0\}$ it follows that $M^\perp \cap S$ is an $(s-m)$-dimensional space. We will now pass to the quotient space $M^\perp /M$.  Observe that an $s$-space $T$ through $M$ is opposite to $S$ if and only if $T/M$ is opposite to $\langle M, M^\perp \cap S \rangle /M$ in this polar space of rank $n-m$. In other words, we need to count the number of $(s-m)$-spaces opposite to a given $(s-m)$-space in $\PS(n-m,e,q)$. This number equals the degree of the relevant opposition graph, which is known to be $q^{(s-m)(2n-s-m+e-(s-m+1)/2)} = q^{\deg(\Phi_m^s)}$, but one can also deduce it from scratch using \Cref{L:advanced_comb} with $j=l=0$ and $k=s-m$.
\end{proof}

\begin{remark}
    We observe that the number of $s$-spaces through $M$ in the above lemma is $\Phi_m^s$. Since the leading coefficient of $\Phi_m^s$ is $1$, the ``majority'' of $s$-spaces through $M$ is opposite to $S$.
\end{remark}

\begin{comment}
    
\begin{lemma} \label{L: hyperplane intersection}
  Consider a (possibly degenerate) polar space whose number of points is $N$, and a hyperplane $H$ of the ambient projective space. Then the number of points $N'$ of the polar space in $H$ satisfies $\deg(N') = \deg(N)-1$.
\end{lemma}
philipp{Not correct for degenerate polar spaces. Take two generators that intersect in an $(n-1)$-space. Then the union of their points is a degenerate polar space, but a hyperplane can contain only one generator. This decimates roughly have the points but is not a saving of a power of $q$.}

\begin{proof}
    Suppose that the polar space is non-degenerate. Then it is well-known (see for example \cite[Proposition 1.5.6]{Valentino_thesis}) that a hyperplane $H$ of the ambient projective space intersects it either in a non-singular polar space whose rank and type depends on the original polar space or in a degenerate polar space of rank 1 lower, same type and radical of dimension 1. In all cases, the equality holds.

    Now suppose that the polar space has radical $R$ of dimension $r > 0$. Then $H$ intersects either $R$ in a codimension 1 space and contains the base (which is a non-degenerate polar space), or it contains $H$ and intersects the base in one of two ways described in the preceding paragraph. In both cases, again the equality holds.
\end{proof}
\end{comment}

\textbf{Notation.} For the remainder of the paper, we say that a non-negative function $f(q)$ is $O(q^d)$ if $f(q) \le c\cdot q^d$ for some constant $c$ that may depend on $n$. We also say that $f(q)$ is $o(q^d)$ if for every positive constant $c$ there is a $q_0$, such that $|f(q)|\le c\cdot q^d$ for all $q\ge q_0$.  Furthermore, we say that $f(q)$ is $\Omega(g(q))$ if $f(q) \geq c\cdot g(q)$ for some positive constant $c$ that may depend on $n$ and for all $q$ bigger than a fixed $q_0$. 
Finally, we say that $f(q)$ is $\Theta(g(q))$ if $c_1\cdot g(q)\leq f(q) \leq c_2\cdot g(q)$ for constants $c_1,c_2$ that may depend on $n$ and for all $q$ bigger than a fixed $q_0$.

\begin{lemma} \label{L: points in X perp}
    Let $x$ be a real number and let $X$ be a non-empty set of $\Omega(q^{x-1})$ points in $\PS(n,e,q)$, where $n\geq 2$ and $e\geq 1$. Then $\langle X \rangle^\perp$ contains $O(q^{N-x})$ (singular) points, where $N= \deg(\Phi_0^1(n,e,q))$. 
\end{lemma}
\begin{proof}
According to \Cref{L:degree_stuff}, we have $N=2n+e-2$.
As $X$ contains  $\Omega(q^{x-1})$ points, we have $\dim(\langle X\rangle)\geq x$.
   We consider the cases independently.\\

    \underline{Case 1: symplectic polar spaces.} We have a symplectic polar space, so $e=1$ and the ambient vector space of $\PS(n,e,q)$ has dimension $2n$. As the alternating form is non-degenerate, we have $\dim(\langle X \rangle^\perp)\leq 2n-x=N+1-x$, which yields that $\langle X \rangle^\perp$ contains $O(q^{N-x})$ singular points.  \\
    
    \underline{Case 2: quadratic or Hermitian spaces.}
    These are the cases where $\PS(n,e,q)$ is parabolic with odd characteristic ($e=1$), elliptic ($e=2$), or Hermitian ($e=3/2$).\\  
    
    Assume that the underlying quadratic or Hermitian form does not vanish identically on $\langle X \rangle^\perp$.
    In the quadratic case, the ambient vector space has dimension $2n+e$. Since the associated polar form is non-degenerate, we have $\dim(\langle X \rangle^\perp)\leq 2n+e-x$. The singular points of $\langle X \rangle^\perp$ form the zero set of a non-zero quadratic form. Hence, $|\langle X \rangle^\perp \cap\PS(n,e,q)|= O(q^{2n+e-x-2})
    =O(q^{N-x})$.\\
    In the Hermitian case, the ambient vector space has dimension $2n+1$. Since the Hermitian form is non-degenerate, we have $\dim(\langle X \rangle^\perp)\leq 2n+1-x$. The singular points of $\langle X \rangle^\perp$ form the zero set of a non-zero Hermitian form. Therefore, $|\langle X \rangle^\perp\cap\PS(n,e,q)|= O(q^{2n-x-1/2})=O(q^{N-x})$.\\
    
Now, suppose that the defining quadratic or Hermitian form vanishes on $\langle X\rangle^\perp$. Then $\langle X\rangle^\perp$ is totally singular, and hence $|\langle X\rangle^\perp\cap\PS(n,e,q)|=O(q^{\dim(\langle X\rangle^\perp)-1})$.
As the points of $\PS(n,e,q)$ contained in $\langle X\rangle$ form a cone whose radical is $\langle X\rangle^\perp$ and whose base is a polar space of rank $n-\dim(\langle X\rangle^\perp)$, we have 
\begin{align} \label{Align: 2.10}
    |\langle X\rangle\cap\PS(n,e,q)|&=\Gauss{\dim(\langle X\rangle^\perp)}{1}+
    q^{\dim(\langle X\rangle^\perp)}
    \Phi_0^1(n-\dim(\langle X\rangle^\perp),e,q). %\\
   % &=
   % O(q^{N-\dim(\langle X\rangle^\perp)}).
\end{align}
%The term $\Gauss{\dim(\langle X\rangle^\perp)}{1}$ is $O(q^{\dim(\langle X\rangle^\perp)-1})$.
First, let us consider the case $\dim(\langle X\rangle^\perp)<n$.
We can use \Cref{L:degree_stuff} (b) and obtain that $\Phi_0^1(n-\dim(\langle X\rangle^\perp),e,q)=O(q^{2n-2\dim(\langle X\rangle^\perp)+e-2})=O(q^{N-2\dim(\langle X\rangle^\perp)})$. So, the first summand of (\ref{Align: 2.10}) is  $O(q^{\dim(\langle X\rangle^\perp)-1})$ and the second summand of (\ref{Align: 2.10}) is $O(q^{N-\dim(\langle X\rangle^\perp)})$. As $\dim(\langle X\rangle^\perp)-1<N-\dim(\langle X\rangle^\perp)$, we get $|\langle X\rangle\cap\PS(n,e,q)|=O(q^{N-\dim(\langle X\rangle^\perp)})$.\\
Now, let us consider the case $\dim(\langle X\rangle^\perp)=n$. Then $\Phi_0^1(n-\dim(\langle X\rangle^\perp),e,q)=0$.  Hence, we get $|\langle X\rangle\cap\PS(n,e,q)|=O(q^{\dim(\langle X\rangle^\perp)-1})$. As $n-1\leq N-n$ for $e\geq 1$, we obtain $|\langle X\rangle\cap\PS(n,e,q)|=O(q^{N-\dim(\langle X\rangle^\perp)})$ in both cases.

We have $X\subseteq \langle X \rangle\cap \PS(n,e,q)$ and as $|X|=\Omega(q^{x-1})$, this gives $x-1\leq N-\dim(\langle X\rangle^\perp)$, which implies the statement.\\

\underline{Case 3: parabolic quadrics in even characteristic.}
Let $R$ be the one-dimensional radical of the associated polar bilinear form. Projection from $R$ induces an isomorphism to a symplectic polar space, which preserves cardinalities and common perpendiculars.
Therefore, the statement follows from Case 1.
\end{proof}

\section{Graph homomorphisms for opposition graphs}
\label{Subsection: Maximal examples}

 Some examples coming from \Cref{E: blow-ups of s-spaces} were already described in Theorem 3.18 of \cite{AlgebraicApproach}. We provide some insight into how one can think of these examples from a graph-theoretic point of view. We also give a concrete mathematical description to the term `blow-up' mentioned in the introduction, which will allow us to prove \Cref{C: classification for flags} and \Cref{thm:main2}. \\

For a finite simple graph $\Gamma$ we use $\alpha(\Gamma)$ to denote the independence number of $\Gamma$. In view of \Cref{thm:upperbound} we have
\begin{align*}
    \frac{|V(\Gamma_J)|}{\alpha(\Gamma_J)}=q^{n+e-1}+1
\end{align*}
for all instances in which the bound of said theorem is tight. Since $\Gamma_J$ is vertex transitive for all $J \subseteq [n]$, we have that the quotient stated above is equal to the fractional chromatic number if the bound of \Cref{thm:upperbound} is tight, see Corollary 7.5.2 in \cite{algebraic_graph_theory}. The bound of \Cref{thm:upperbound} is tight for $\Gamma_{[n]}$ and $\Gamma_1$. %Since pairwise collinear points span a totally isotropic subspace, the maximum independent sets of $\Gamma_1$ are precisely the sets of points of a fixed generator. \red{this sentence is a duplicate now?}

We can define a graph homomorphism $\varphi$ from $\Gamma_{[n]}$ to $\Gamma_1$ if we map every chamber of $\Gamma_{[n]}$ to its point. Adjacent vertices are opposite chambers, and opposite chambers contain opposite points. Therefore, this homomorphism indeed preserves adjacency. It follows immediately that the preimage of a maximum independent set of $\Gamma_1$ under $\varphi$ is a maximum independent set of $\Gamma_{[n]}$, see for example \cite[Lemma 7.5.4]{algebraic_graph_theory}. 
The description of maximum independent sets of $\Gamma_1$ (see \Cref{E: EKR for s-spaces}) hence gives a description of the maximum independent sets of $\Gamma_{[n]}$ that are the preimage of a maximum independent set of $\Gamma_1$ under $\varphi$.
%\rem{Since the structure of maximum EKR-sets of $\Gamma_1$ is known, we get a description of the maximum EKR-sets of chambers that are based on a maximum EKR-set of points.} \red{this last sentence is not clear to me?}
%\philipp{Do you think this is better?} \\

In general, the procedure outlined above works for any $J \subseteq [n]$.

\begin{notation}
    Let $\EKR_J$ be an EKR-set of flags of type $J$. If $\EKR$ is the set of all chambers that contain a flag of type $J$ in $\EKR_J$, we call $\EKR$ the \emph{blow-up} of $\EKR_J$.
\end{notation}

\begin{remark} \label{R: blow-up}
Let $\EKR_J$ be an EKR-set of flags of type $J$. Then the blow-up of $\EKR_J$ is a ratio-sharp EKR-set of chambers if and only if $\EKR_J$ is ratio-sharp.
\end{remark}

For $\Gamma_n$ we also know that the ratio bound is tight (see \cite{stanton_ekr_chevalleygroups}) and the maximum independent sets are classified in \cite{EKR_polarspace_generators}.

%Move to the end of the paper?
With this in mind, we now prove \Cref{C: classification for flags} and \Cref{thm:main2} given \Cref{thm:main}. First, we need a technical lemma.

%We can apply this procedure for other choices of $J \subseteq [n]$. For $\Gamma_n$ we also know that the ratio bound is tight (see \cite{stanton_ekr_chevalleygroups}) and the maximum independent sets are classified in \cite{EKR_polarspace_generators}. With this approach in mind, we can easily prove \Cref{C: classification for flags} given \Cref{thm:main}.

%\begin{proof}[Proof of \Cref{C: classification for flags}]
% In general, the procedure outlined above works for any $J \subseteq [n]$, a maximum EKR-set of flags in $\Gamma_J$ meeting the ratio bound, would produce maximum EKR-sets of chambers.
% \philipp{Put this maybe in a remark and cite it at the end of section 5, no calculation with $B$ needed (also in the prof of 1.7)}
% Since the latter are classified by \Cref{thm:main}, it follows immediately that the preimage of any example in $\Gamma_J$ must correspond, through the appropriate homomorphism, to one of the examples in \Cref{E: blow-ups of s-spaces} for $s\in \{1,n\}$. This is clearly only possible whenever $1 \in J$ or $n \in J$.    
%\end{proof}

\begin{lemma}
Let $i\in \{1,n\}$ with $n\geq 4$. Furthermore, let $\EKR_i$ be a largest EKR-set of $i$-spaces and let $J\subseteq [n]$ with 
%$1,n\notin J$.
$i\notin J$. Then $\EKR_i$ contains two $i$-spaces $I_1$ and $I_2$ that can be extended to flags $F_1$ and $F_2$ of type $\{i\}\cup J$ such that $F_1\setminus\{I_1\}$ and $F_2\setminus \{I_2\}$ are opposite.
\end{lemma}

\begin{proof}
    First, let $J=\{t_1,\ldots,t_f \}$ with $t_1<\ldots<t_f$. We start with the case $i=1$ and prove the claim via induction on $f$. Let $P_1$ and $P_2$ be any two distinct points in $\EKR_1$. For any dimension $s$ with $s>1$, there is an $s$-space $S$ incident with $P_1$, such that $S^\perp\cap P_2=\{0\}$. Since $P_2$ is a point not in $S^\perp$, there exists an $s$-space incident with $P_2$ that is opposite to $S$. This shows the result for $f=1$.

    Now assume $f>1$. By the induction hypothesis we can find two flags $F_1'$ and $F_2'$ of type $\{ 1,t_1,\ldots,t_{f-1}\}$ with $F_1'\setminus\{P_1\}$ and $F_2'\setminus\{P_2\}$ being opposite. Let $M_i$ be the $t_{f-1}$-space of $F_i'$ for $i=1,2$. In particular $M_1\cap M_2^\perp=\{0\}$. Since there are opposite chambers in $\PS(n,e,q)$ containing $M_1$ and $M_2$ respectively, there exist $t_f$-spaces $S_1$ incident with $M_1$, and $S_2$ incident with $M_2$, such that $S_1^\perp \cap S_2=\{0\}$.
    By taking $F_1:=F_1'\cup \{ S_1\}$ and $F_2:=F_2'\cup \{S_2 \}$ we obtain the result. 

    Now we consider the case $i=n$. Note that in the constructions described in \Cref{E: EKR for s-spaces} (b)-(d) there are always two generators that meet only in a point. Let $\pi_1$ and $\pi_2$ be two such generators. For any dimension $m<n$ there exists an $m$-space $M$ incident with $\pi_1$ such that $M^\perp$ meets $\pi_2$ only in dimension $n-m$. The statement follows by an inductive argument analogous to the case $i=1$, we leave the details to the reader. 
\end{proof}

\begin{proof}[Proof of \Cref{C: classification for flags} and \Cref{thm:main2}]
Let $J\subseteq [n]$ and suppose that $\EKR_J$ is a ratio-sharp EKR-set of flags of type $J$. By \Cref{R: blow-up} the blow-up $\EKR$ of $\EKR_J$ is a ratio-sharp EKR-set of chambers. 
In \Cref{C: classification for flags} and \Cref{thm:main2} we have that $q$ is large enough, so we can apply \Cref{thm:main} and we have that there is a ratio-sharp EKR-set $\EKR_i$ of $i$-spaces with $i\in \{1,n\}$, such that $\EKR$ is the blow-up of $\EKR_i$.

First assume $i\notin J$. The previous lemma implies that there are two chambers in $\EKR$ whose flags of type $J$ are opposite. This stands in contradiction to $\EKR$ being a blow-up of $\EKR_J$. As $1,n\notin J$ implies $i\notin J$, this proves \Cref{thm:main2}.

We have $i\in J$. Now assume that $\EKR_J$ is not the set of all flags of type $J$ that contain an $i$-space of $\EKR_i$. Since $\EKR$ is a blow-up of $\EKR_J$ and $i\in J$ this implies that $\EKR$ is not the blow-up of $\EKR_i$, which is a contradiction.
Hence  $\EKR_J$ is the set of all flags of type $J$ that contain an $i$-space of $\EKR_i$. This proves \Cref{C: classification for flags}.
%This proves \Cref{thm:main2}.
%Now assume that $i\in J$. 
%First assume that $i\in J$. 
%Furthermore, assume that $\EKR_J$ is not the set of all flags of type $J$ that contain an $i$-space of $\EKR_i$. Since $\EKR$ is a blow-up of $\EKR_J$ and $i\in J$ this implies that $\EKR$ is not the blow-up of $\EKR_i$, which is a contradiction.
%
%Now assume that $i\notin J$, but $n+1-i\in J$. This proves \Cref{C: classification for flags}.
%
%Finally, let us assume $1,n\notin J$. The previous lemma implies that there are chambers in $\EKR$ whose flags of type $J$ are opposite. This stands in contradiction to $\EKR$ being a blow-up of $\EKR_J$. This proves \Cref{thm:main2}.
\end{proof}

%\begin{proof}[Proof of \Cref{thm:main2}]
%Let $\EKR_J$ be a ratio-sharp EKR-set of flags of type $J$. By \Cref{R: blow-up} the blow-up $\EKR$ of $\EKR_J$ is a ratio-sharp EKR-set of chambers. By \Cref{thm:main} we have that there is a ratio-sharp EKR-set $\EKR_i$ of $i$-spaces with $i\in \{1,n\}$, such that $\EKR$ is the blow-up of $\EKR_i$. Let us assume $i\notin J$. The previous lemma implies that there are chambers in $\EKR$ whose flags of type $J$ are opposite. This stands in contradiction to $\EKR$ being a blow-up of $\EKR_J$.
%\end{proof}

%\red{So then I suppose you cannot be happy with this proof either: for suppose $1 \notin J$, you blow up and you magically find the blow-up of a max EKR-set of points. How do you know this does not happen?}
%\philipp{Agreed. I tried to write a clean argument above.}

We turn our attention back to the graph homomorphism between $\Gamma_J$ and $\Gamma_{[n]}$ for general $J \subseteq [n]$.
The relation between them provides us even more information by `linearizing' the graph homomorphism. To be precise, we can construct an equitable partition, see e.g.\  \cite[Section 2.2]{GodsilMeagher} for more background, of $\cF_{[n]}$ by partitioning the chambers according to the flag of type $J$ they contain. With this partition, there is a corresponding quotient matrix $Q$, indexed by $\cF_J$, and incidence matrix $M$, whose rows and columns are indexed by $\cF_{[n]}$ and $\cF_J$ respectively, so that if $A_{[n]}$ is the adjacency matrix of $\Gamma_{[n]}$, we have the equality
\begin{align}\label{eq:quotientmatrixrelation}
    A_{[n]}M = MQ.
\end{align}

In fact, by Brouwer's result \cite[Corollary 3.2]{brouwer_eigenvalues}, we know that the quotient matrix is $Q=q^\ell\cdot A_J$, where $\ell$ is the weighted length with respect to the building parameters of the longest word in the corresponding Coxeter group $W_{[n]\setminus J}$ and $A_J$ is the adjacency matrix of $\Gamma_J$, see \Cref{L: basicnumbers}. In any case, we find a relation between $A_{[n]}$ and (a scalar multiple of) $A_J$, by the matrix $M$. 

Observe that given an eigenvector $v$ for $Q$ with eigenvalue $\lambda$, it follows from \eqref{eq:quotientmatrixrelation} that $Mv$ is an eigenvector for $A_{[n]}$ with the same eigenvalue. This relies in particular on the fact that $M$ defines an injective map $\mathbb{R}^{|\cF_J|} \to \mathbb{R}^{|\cF_{[n]}|}$. As a result, we obtain the following lemma.

\begin{lemma}\label{lem:eigenvectorsthroughM}
    Let $J \subseteq [n]$. Suppose that $v,w \in \mathbb{R}^{|\cF_J|}$ are eigenvectors of $Q$ corresponding to different eigenvalues. Then $Mv$ and $Mw$ are orthogonal eigenvectors of $A_{[n]}$.
\end{lemma}
\begin{proof}
    By our observation, $Mv$ and $Mw$ are eigenvectors for $A_{[n]}$ for different eigenvalues and since $A_{[n]}$ is symmetric, it follows that they are orthogonal.
\end{proof}

\section{Antidesigns}\label{sec:antidesigns}

In this section, we will construct antidesigns for ratio-sharp EKR-sets of flags of type $J$ for $J \subseteq [n]$. By equality in the ratio bound, we know that their characteristic vectors are contained in the sum of eigenspaces corresponding to the largest and smallest eigenvalue. When $e > 1$ or $e \in \{0,1\}$ and $n$ even, we have precise control of how the eigenspaces of these eigenvalues look like by \cite{AlgebraicApproach2}. Remarkably, the proofs in this section do not depend on this precise description. It was however instrumental in finding the antidesigns, even though the current iteration of the proof no longer mentions it.

\begin{de}\label{def:antidesign}
Let $J \subseteq [n]$. An \emph{antidesign} of $\Gamma_J$ is a vector in $\mathbb{R}^{|\cF_J|}$ orthogonal to the eigenspace corresponding to the smallest eigenvalue of $\Gamma_J$.
\end{de}

We also view a vector $v \in \mathbb{R}^{|\cF_J|}$ as a map $\cF_J \to \mathbb{R}$ from the set of all flags of type $J$ to $\mathbb{R}$; then $v(F)$ is the entry of $v$ in the position corresponding to the flag $F$. For a subset $X$ of flags we define its \emph{characteristic vector} $\charvec_X$ by $\charvec_X(F)=1$, if $F\in X$, and $\charvec_X(F)=0$, if $F$ is a flag not in $X$. We shall denote the all ones vector of $\mathbb{R}^{|\cF_J|}$ by $\allones$, which is the characteristic vector of  $\cF_J$. The following well-known fact relates the antidesign to the ratio-sharp EKR-set and allows us to obtain local geometric information about the latter.

\begin{prop}\label{P: intersection_F_with_antidesing}
If $\EKR_J$ is a ratio-sharp independent set of $\Gamma_J$ and $v$ is an antidesign, then
\begin{align*}
\charvec_{\EKR_J}^\top v=\frac{(\allones^\top \allones_{\EKR_J})(\allones^\top v)}{\allones^\top \allones}.
\end{align*}
\end{prop}

The proof is well known. In the interest of keeping the paper self-contained, we give the short argument.

\begin{proof}
If $E$ is the eigenspace of the smallest eigenvalue of $\Gamma_{J}$, then $\charvec_{\EKR_J}\in\langle \allones \rangle+E$, that is $\charvec_{\EKR_J}=k\allones+e$ for some $k\in\mathbb{R}$ and $e\in E$. We write $v=\ell \allones+w$ with $\allones^\top w=0$. Since $e^\top v=0$, $\allones^\top \charvec_{\EKR_J}=k\allones^\top \allones$ and $\allones^\top v=\ell \allones^\top \allones$, then
\begin{align*}
\charvec_{\EKR_J}^\top v =(k\allones+e)^\top(\ell \allones+w)=k\ell \allones^\top \allones =\frac{(\allones^\top\charvec_{\EKR_J})(\allones^\top v)}{\allones^\top \allones}.  & \qedhere
\end{align*}
\end{proof}

\begin{remark}
If $\{g_1,\ldots,g_m\}$ is a set of spanning vectors of the smallest eigenspace it would suffice to show that $v^\top g_i=0$ for all $1\leq i\leq m$, in order to show that $v$ is an antidesign. This approach is in line with \cite{klaus_jesse_philipp}. A set of spanning vectors of $\Gamma_{[n]}$ is described in \cite{AlgebraicApproach2}, however, such a description of the eigenspaces is only known for $e>1$ or $e=1$ and $n$ even, so this does not quite cover all cases we are interested in. The dependency on the eigenspace is clearly a downside of this approach. Furthermore, this approach requires lengthy combinatorial arguments. Finally, it is a priori not clear how one can actually find antidesigns using this method.
\end{remark}

The relation between eigenvectors of $\Gamma_J$ and $\Gamma_{[n]}$ discussed in \Cref{lem:eigenvectorsthroughM} allows us to lift antidesigns.

\begin{lemma}\label{lem:transferantidesign}
    Let $s \in [n]$. Suppose that $v$ is an antidesign of $\Gamma_s$, then $Mv$ is an antidesign of $\Gamma_{[n]}$.
\end{lemma}
\begin{proof}
Let $w$ be an eigenvector of the adjacency matrix $A_s$ of $\Gamma_s$ with eigenvalue $\theta$. As $A_{[n]}M=Mq^\ell A_s$, where $\ell$ is the weighted length with respect to the building parameters of the longest word in the corresponding Coxeter group $W_{[n]\setminus\{s\}}$, we have that $Mw$ is an eigenvector of $A_{[n]}$ with eigenvalue $q^\ell\theta$. It is easily checked that $q^\ell \lambda_s=\lambda_{[n]}$, as $\ell=\deg(z_sz_{n-s}\Phi_s^n)$.

    By definition, the decomposition of $v$ as a linear combination of eigenvectors does not contain any eigenvector corresponding to the smallest eigenvalue of $\Gamma_s$. By \Cref{lem:eigenvectorsthroughM} this property remains true for $Mv$ and hence it is orthogonal to the eigenspace corresponding to the smallest eigenvalue of $\Gamma_{[n]}$.
\end{proof}

Now we will study how one can generate antidesigns and construct subspace-based antidesigns for $\Gamma_{[n]}$ and $\Gamma_s$ with $s\in[n]$. The former will be used to show that each maximum EKR-set of chambers is always a homomorphic preimage of a maximum EKR-set in $\Gamma_s$, whereas the latter will be used to prove that these only exist when $s \in \{1,n\}$. As usual, we assume that $e\ge 1$.\\

As a first step, consider the adjacency matrix $A_J$ of the graph $\Gamma_J$ with $J\subseteq [n]$. Let $\lambda_J$ be the smallest eigenvalue of $\Gamma_J$. If $\chi$ is an eigenvector of $A_J$ corresponding to the smallest eigenvalue, we observe that
$$(A_J-\lambda_J I)\chi = 0.$$

In other words, every row of $A_J-\lambda_J I$ is orthogonal to the $\lambda_J$-eigenspace of $A_J$ and is hence an antidesign for $\Gamma_J$. To be precise, looking at the rows indexed by the flags of type $J$, we get an explicit description.

\begin{lemma}
Let $F$ be a fixed vertex of $\Gamma_J$. Then the following is an antidesign for $\Gamma_J$.
\begin{align*}
\chi_F(G):=\begin{cases}
-\lambda_J& \text{if $F=G$,}
\\
1 &  \text{if $F$ and $G$ are opposite,}
\\0 & \text{otherwise.}
\end{cases}
\end{align*}
\end{lemma}

If $F$ is a flag that contains only one subspace $S$, we also write $\chi_S$. Let $d_J$ be the degree, i.e.\ the valency of $\Gamma_J$. It was already observed in \cite[Theorem 3.15]{AlgebraicApproach} that for the given values of $e$ and $n$, we have $d_J=-\lambda_Jq^{n+e-1}$.

\begin{cor} \label{C: intersection basic}
Let $\EKR_J$ be a ratio-sharp EKR-set for flags of type $J$ of $\PS(n,e,q)$. Furthermore, let $F$ be a fixed flag of type $J$. Then $\allones_{\EKR_J}^\top \chi_F=-\lambda_J$.
\end{cor}

\begin{proof}
    There is precisely one flag $G$ with $G=F$. Furthermore, there are exactly $d_J$ many flags $G$ with $G$ and $F$ being opposite. In view of \Cref{P: intersection_F_with_antidesing} this yields
\begin{align*}
\begin{split}
\allones_{\EKR_J}^\top \chi_F &=\frac{(\allones^\top \allones_{\EKR_J})(\allones^\top \chi_F)}{\allones^\top \allones}\\
&= \frac{\left(\frac{|\cF_J|}{q^{n+e-1}+1}\right) \left( -\lambda_J +d_J \right)}{|\cF_J|} \\
&= \frac{d_J-\lambda_J}{q^{n+e-1}+1} \\
&=-\lambda_J.
\qedhere
\end{split}
\end{align*}
\end{proof}

\begin{remark}
The eigenspace of the smallest eigenvalue $\lambda_J$ is the null space of $A_J-\lambda_JI$ and the rowspace is its orthogonal complement. Therefore, every antidesign of $\Gamma_J$ is a linear combination of the antidesigns $\chi_F$.
\end{remark}

Next, we examine a subspace-based antidesign for chambers. Recall that we denote $\lambda_J$ as $\lambda_s$ if $J=\{s\}$ and 
$$ -\lambda_s:= q^{s(2n+e-1/2-3s/2)-n-e+1}$$
as well as 
$$-\lambda_{[n]}=q^{(n-1)(n+e-1)}.$$

\begin{remark} \label{R: degree eqn}
In view of \Cref{L:degree_stuff}, we have
    $\deg(-\lambda_s)+\deg(z_sz_{n-s}\Phi_s^n)=\deg(-\lambda_{[n]}).$
\end{remark}

\begin{lemma} \label{N: SubspaceAntidesignChambers}
Let $s \in [n]$ and $S \in \cF_s$, then the function $v_S:\cF_{[n]}\to \mathbb{R}$ defined by
\begin{align*}
v_S(C)=\begin{cases}
-\lambda_s & \text{if $S=C_s$,}
\\
1 &  \text{if $S\cap C_{s}^\perp=\{0\}$,}
\\0 & \text{otherwise,}
\end{cases}
\end{align*}
is an antidesign for $\Gamma_{[n]}$. 
\end{lemma}

\begin{proof}
    Let $S$ be a fixed $s$-space. Recall that the $s$-space of a chamber $C$ is denoted as $C_s$. We will show that $\sum\limits_{C_s=S}\chi_C$ is a scalar multiple of $v_S$. 
    
    For any chamber $B$ with $B_s=S$ we get a contribution of $-\lambda_{[n]}$ from $\chi_B$. In addition, $B$ is not opposite to any chamber $C$ with $C_s=S$, so this is the only non-zero contribution. Therefore, we get $\sum\limits_{C_s=S}\chi_C(B)=-\lambda_{[n]}=q^{(n-1)(n+e-1)}$ if $B_s=S$.

    Now, let $B$ be a chamber such that $B_s$ is not equal to, but also not opposite to $S$. In this case $B$ will be non-opposite to any chamber $C$ with $C_s=S$. Hence, we get $\sum\limits_{C_s=S}\chi_C(B)=0$.

    Finally, let $B$ be a chamber with $B_s$ opposite to $S$. According to \Cref{L: basicnumbers}, the number of chambers $C$ with $C_s=S$ opposite to $B$ is precisely $q^{\deg(z_sz_{n-s}\Phi_s^n)}$.

    The assertion follows from \Cref{R: degree eqn}.
\end{proof}

It is not a coincidence that we see $-\lambda_s$ in the antidesign for $\Gamma_{[n]}$. The fact that $v_S$ is an antidesign also follows directly from \Cref{lem:transferantidesign} applied to $\chi_S$.

\begin{cor} \label{C: subspace_intersection}
Let $\EKR$ be a ratio-sharp EKR-set of chambers of $\PS(n,e,q)$.
We have 
$$v_S^\top \allones_\EKR %= z_sz_{n-s} \Phi_s^n q^{n(2s-1)+e(s-1)-s(3s+1)/2+1}
=-\lambda_s z_sz_{n-s} \Phi_s^n .$$
\end{cor}
\begin{proof}
    In the proof of \Cref{N: SubspaceAntidesignChambers} we have shown that 
    $$v_S=q^{-\deg(z_sz_{n-s}\Phi_s^n)}\cdot\sum\limits_{C_s=S}\chi_C.$$
    The number of chambers $C$ with $C_s=S$ is, according to \ref{L: extend flags}, precisely $z_sz_{n-s}\Phi_s^n$. This implies 
    \begin{align*}
    \begin{split}
        v_S^\top \allones_\EKR&= q^{-\deg(z_sz_{n-s}\Phi_s^n)} \cdot \sum\limits_{C_s=S}\chi_C^\top \allones_\EKR\\
        &= q^{-\deg(z_sz_{n-s}\Phi_s^n)} \cdot (-z_sz_{n-s}\Phi_s^n\lambda_{[n]})\\
        &= -z_sz_{n-s}\Phi_s^n\lambda_{s}.
        \qedhere
    \end{split}
    \end{align*}
\end{proof}

By summing over antidesigns $\chi_S$, we also obtain a subspace-based antidesign for $\Gamma_s$.

\begin{lemma} \label{L:m-spaces_m<s}
  Let $1\leq m<s$ be integers and let $M$ be an $m$-subspace. Then the function $v_M:\cF_s\to \mathbb{R}$ defined by
\begin{align*}
v_M(T):=\sum\limits_{M \subseteq S}\chi_S(T) =\begin{cases}
-\lambda_s & \text{if $M \subseteq T$,}
\\
q^{\deg(\Phi_m^s)} &  \text{if $M \cap T^\perp =\{0\}$ } \\
0 & \text{otherwise,}
\end{cases}
\end{align*}
is an antidesign for $\Gamma_s$.
\end{lemma}

\begin{proof}
    The value of $v_M(T)$ depends on the mutual position of $T$ and $M$. If $M\subseteq T$, then we only get a contribution of $-\lambda_s$ from $\chi_T$ since $T$ is not opposite to any other $s$-space on $M$. If $T$ does not contain $M$, but $T^\perp$ meets $M$ in at least a point then $T$ is non-opposite to any subspace on $M$ so the value in question is $0$ in this case. Now, if $T^\perp\cap M=\{0\}$ the number of $s$-spaces incident with $M$ and opposite to $T$ is, according to \Cref{L: heavy power m-space}, exactly $q^{\deg(\Phi_m^s)}$. 
\end{proof}

\begin{cor} \label{C: Intersection m-spaces m<s}
Let $\EKR_s$ be a ratio-sharp EKR-set of $s$-spaces of $\PS(n,e,q)$ and let  $1\leq m<s$ be integers such that $M$ is an $m$-subspace. We have 
$$v_M^\top \allones_{\EKR_s} = -\lambda_s \Phi_m^s .$$
\end{cor}

\begin{proof}
    According to \Cref{C: intersection basic}, we have $\chi_S^\top \allones_{\EKR_s}=-\lambda_s$. For the antidesign $v_M$ we sum over all the $s$-spaces incident with $M$ and according to \Cref{L: extend subspaces} there are precisely $\Phi_m^s$ such $s$-spaces. Since $v_M=\sum\limits_{M \subseteq S}\chi_S$, we have  $v_M^\top \allones_{\EKR_s}=\sum\limits_{M \subseteq S} \chi_S^\top \allones_{\EKR_s} $. The result follows. 
\end{proof}

\begin{remark}
Ovoids, spreads, and more generally $m$-systems \cite{m-systems} are other known antidesigns \cite[Theorem 4.4.14]{Vanhove_PhD}, albeit not usually described with this term.
Let $\mathcal{M}$ be an $m$-system.
In our notation the following map $v_\mathcal{M}$ from the set of chambers of $\PS(n,e,q)$ to $\mathbb{R}$ is an antidesign:
\begin{align*}
v_\mathcal{M}(C)=\left\{\begin{array}{ll}
1 & \text{if $C_m\in \mathcal{M}$,}
\\
0 &  \text{if $C_m\notin \mathcal{M}$.}
\end{array}
\right.
\end{align*}
It is straightforward to prove that this is a linear combination of all $\chi_C$ with $C_m\in \mathcal{M}$.
\end{remark}

 \section{Maximum EKR-sets of chambers} \label{Section: Classification of ekr}

In this section, we always consider polar spaces $\PS(n,e,q)$ with $e\ge 1$. Our goal is to show that a maximum EKR-set of chambers is a homomorphic preimage of a maximum EKR-set of $s$-spaces. In \Cref{S:EKR_s-spaces} we rule out all $s$ with $s\notin \{1,n\}$. The approach we use is somewhat similar to \cite{klaus_jesse_philipp}.

\begin{notation} Let $\EKR$ be an EKR-set of chambers of $\PS(n,e,q)$.
\begin{enumerate}[1.]
\item 
For every subspace $S$ of $\PS(n,e,q)$ the weight of $S$ (with respect to $\EKR$) is the number of chambers of $\EKR$ that contain $S$. We call $S$ \emph{heavy} (with respect to $\EKR$) if all chambers that contain $S$ belong to $\EKR$; otherwise, we call $S$ \emph{light}.
\item A chamber is called \emph{light} if all its subspaces are light, and otherwise it is called \emph{heavy}.
\end{enumerate}
\end{notation}

The following lemma is prominently featured in almost all proofs of this section. Note that an integral part of the proof of this lemma is the subspace-based antidesign. As before, we use the notation $\lambda_s=-q^{s(2n+e-1/2-3s/2)-n-e+1}$. Recall also that the size of a maximum EKR-set of chambers in $\PS(n,e,q)$ is $\Phi_1^n z_n$ and $\deg(\Phi_1^nz_n)=(n-1)(n+e-1)$.

\begin{lemma} \label{L: key for the proof}
Let $s \in [n]$, $S$ be a subspace of dimension $s$, and $\EKR$ a ratio-sharp EKR-set of chambers of $\PS(n,e,q)$. Considering the chambers $C \in \EKR$, let $X$ be the number of these with $C_s=S$, let $Y$ be the number of these with $C_s^\perp \cap S=\{0\}$, and let $Z$ be the number of these with $C_s\neq S$ and $C_s^\perp \cap S\neq \{0\}$.
\begin{enumerate}[(a)]
\item We have $Y=\lambda_s (X-z_sz_{n-s} \Phi_s^n)$.
\item If $S$ is heavy, then $X=z_sz_{n-s}\Phi_s^n$ and $Z=\Phi_1^n z_n- z_sz_{n-s}\Phi_s^n$.
\item If $S$ is light, we have $X\le z_sz_{n-s}\Phi_s^n-q^{\deg(z_sz_{n-s}\Phi_s^n)}$ and furthermore we have $$Z\le \Phi_1^nz_n -z_sz_{n-s}\Phi_s^n+ (\lambda_s+1) q^{\deg(z_sz_{n-s} \Phi_s^n)}. $$
\end{enumerate}
\end{lemma}
\begin{proof}
We have $X+Y+Z=|\EKR|=\Phi_1^nz_n$. From the definition of $v_S$ (see \Cref{N: SubspaceAntidesignChambers}) 
we see that $\charvec_\EKR^\top v_S=-\lambda_s X+Y$ so \Cref{C: subspace_intersection} implies 
\begin{align*}
	\begin{split}
		Y&=-z_sz_{n-s} \Phi_s^n \lambda_s + \lambda_s X\\
		&=-\lambda_s (z_sz_{n-s} \Phi_s^n -X).
	\end{split}
\end{align*}
Since $Y\ge 0$, it follows that $X\le z_sz_{n-s}\Phi_s^n$. If equality holds, then $Y=0$, we are in situation (b) and
\begin{align*}
	\begin{split}
		Z&= |\EKR|-X\\
		&= \Phi_1^n z_n- z_sz_{n-s}\Phi_s^n.
	\end{split}
\end{align*} 
Now consider the situation when $X<z_sz_{n-s}\Phi_s^n$. Then $Y$ is positive, so there exists a chamber $C\in \EKR$ such that $C_s^\perp \cap S=\{ 0\}$. 
\Cref{L: basicnumbers} shows that the number of chambers that contain $S$ and are opposite to $C$ is $q^{\deg(z_sz_{n-s} \Phi_s^n)}$. Since the total number of chambers containing $S$ is $z_sz_{n-s}\Phi_s^n$, it follows that $X\le z_sz_{n-s}\Phi_s^n-q^{\deg(z_sz_{n-s} \Phi_s^n)}$. 
	Furthermore, we get
	\begin{align*}
		\begin{split}
			Z&=|\EKR|-X-Y\\
			&=\Phi_1^nz_n-X+ \lambda_s (z_sz_{n-s} \Phi_s^n -X)\\
			&=\Phi_1^nz_n+\lambda_s z_sz_{n-s}\Phi_s^n + (-\lambda_s -1)X\\
			%&\le \Phi_1z_n-q^{s(2n-s)-n+(s-1)(e-1-\frac{s}{2})} z_sz_{n-s}\Phi_s
                 &\le \Phi_1^nz_n+\lambda_s z_sz_{n-s}\Phi_s^n + (-\lambda_s -1)(z_sz_{n-s}\Phi_s^n-q^{\deg(z_sz_{n-s} \Phi_s^n)})\\
                 &= \Phi_1^nz_n -z_sz_{n-s}\Phi_s^n+ (\lambda_s+1) q^{\deg(z_sz_{n-s} \Phi_s^n)} 
		\end{split}
	\end{align*}
	and we are in situation (c).
\end{proof}

\begin{cor} \label{C: light degree} 
In case (c) of \Cref{L: key for the proof}, we have
\begin{enumerate} [(a)]
	\item  $X=O(q^{\deg(z_sz_{n-s}\Phi_s^n)-1})$, and %\deg(X)\le \deg(z_sz_{n-s} \Phi_s^n)-1$, and 
	\item  $Z=O(q^{(n-1)(n+e-1)-1})$. %$\deg(Z)\le (n-1)(n+e-1)-1$.
\end{enumerate}
\end{cor}
\begin{proof}
Part (a) follows directly from the fact that the leading coefficient of $z_sz_{n-s}\Phi_s^n$ is $1$.
For part (b) observe that $\deg(\Phi_1^nz_n)=\deg(-\lambda_{[n]})$. The statement follows from \Cref{R: degree eqn}. 
\end{proof}

 First, we consider the case where $\EKR$ contains a light chamber.

\begin{prop} \label{P: light chamber}
	Let $\EKR$ be a ratio-sharp Erd\H os-Ko-Rado set of chambers of $\PS(n,e,q)$. 
    If $q$ is large enough in terms of $n$, then $\EKR$ cannot contain a light chamber.
 %   Assume that $\EKR$ contains a light chamber. Then,  $|\EKR|$ is $O(q^{(n-1)(n+e-1)-1})$.
\end{prop}

%\red{Maybe this statement is a bit awkward: if $\EKR$ is ratio-sharp and blabla then $|\EKR|$ is $O(q^{(n-1)(n+e-1)})$? That's a contradiction by itself. Perhaps it is better to say, if $\EKR$ is a ratio-sharp EKR set of chambers, it cannot have a light chamber.}

\begin{proof}
Let $\EKR$ be a ratio-sharp EKR-set of chambers and assume that $\EKR$ contains a light chamber $C$.	Every chamber $B$ in $\EKR$ has to satisfy $B_{s}^\perp \cap C_s\neq \{0 \}$ for some $s \in [n]$. Let $X_s$ be the number of chambers $B$ with $B_s=C_s$ and let $Z_s$ be the number of chambers $B$ with $B_s\neq C_s$ and $B_{s}^\perp \cap C_s\neq \{0 \}$. Then
	\begin{align} \label{A: |F| light chamber}
		|\EKR|\leq \sum\limits_{s=1}^{n} (X_s+Z_s).
	\end{align}
	Recall that $\deg(|\EKR|)=(n-1)(n+e-1)$.
	\Cref{C: light degree} yields 
	\begin{align*}
		\begin{split}
			Z_s&=O(q^{(n-1)(n+e-1)-1}),\\                    %\deg(Z_s)& \le (n-1)(n+e-1)-1, \\
			X_s&=O(q^{\deg(z_sz_{n-s}\Phi_s^n)-1}).       %\deg(X_s)&\le \deg(z_sz_{n-s} \Phi_s^n)-1.\\
		\end{split}
	\end{align*}

\Cref{R: degree eqn} implies that $\deg(z_sz_{n-s} \Phi_s^n)<\deg(\Phi_1^nz_n)$, hence $X_s=O(q^{(n-1)(n+e-1)-1})$. This implies that $|\EKR|$ is $O(q^{(n-1)(n+e-1)-1})$ and the statement follows.	
\end{proof}

Now we consider the case where every chamber is heavy. We start with some observations on heavy subspaces.

\begin{lemma}
Let $\EKR$ be a maximum EKR-set of chambers in $\PS(n,e,q)$.
\begin{enumerate} [(a)]
    \item Let $S_1$ and $S_2$ be two heavy subspaces. Then $S_1\cap  S_2^\perp\neq \{0 \}$.
    \item Let $S$ be a heavy subspace. Then every chamber $C$ in $\EKR$ satisfies $C_s^\perp \cap S\neq \{ 0\}$
\end{enumerate}
\end{lemma}
\begin{proof}
    For (a) assume that this is not the case. Then we can extend $S_1$ and $S_2$ to two chambers in $\EKR$ that are opposite, since $S_1$ and $S_2$ are heavy. However, $\EKR$ does not contain opposite chambers so this is a contradiction. Similarly, we can assume that (b) is false. Then $S$ can be extended to a chamber that is opposite to $C$, which is a contradiction.	
\end{proof}

\begin{lemma} \label{L: heavy dimension}
	Assume that a ratio-sharp EKR-set $\EKR$ contains a chamber with the property that its $s$-subspace $S$ is light for some $s \in [n]$. 
    Then the number of chambers in $\EKR$ whose $s$-subspace is heavy is $O(q^{(n-1)(n+e-1)-1})$.
\end{lemma}
\begin{proof}
	Let $S$ be a light $s$-space. Then \Cref{L: key for the proof} and \Cref{C: light degree} imply that the number $Z$ of chambers $C$ that contain an $s$-space $C_s$ with $C_s^\perp \cap S \neq \{ 0\}$ satisfies $\deg(Z)\le (n-1)(n+e-1)-1$. Since for every heavy $s$-space $C_s$ we have that $C_s^\perp$ meets $S$, this number $Z$ includes all the chambers whose $s$-space is heavy. 
\end{proof}

\begin{cor} \label{C: chambers - blow-ups}
Let $\EKR$ be a ratio-sharp EKR-set of chambers. For sufficiently large $q$ (in terms of $n$) there exists an $s \in [n]$ such that every chamber in $\EKR$ contains a heavy $s$-subspace.
\end{cor}

\begin{proof}
	Assume that this is not the case. Then for every $s \in [n]$ there is a chamber whose $s$-space is light. By \Cref{P: light chamber} and \Cref{L: heavy dimension} we get that $|\EKR|$ is at most the sum of $n$ numbers that are all $O(q^{(n-1)(n+e-1)-1})$. The statement follows, as this stands in contradiction to the size of $\EKR$ being $\Theta(q^{(n-1)(n+e-1)})$. 
\end{proof}

The last corollary states, together with \Cref{P: light chamber} that for $q$ large enough there are in some sense only $n$ possibilities left to construct a maximum EKR-set of chambers: we have to start with a maximum EKR-set of $s$-spaces and consider its blow-up. These are precisely the examples considered in \Cref{E: blow-ups of s-spaces}. In the next section, we will show that for $s\notin \{1,n\}$ the resulting set of chambers is not a maximum one. In view of \Cref{R: blow-up} it suffices to show that for $s\notin \{1,n\}$ a maximum EKR-set of $s$-spaces is not ratio-sharp.

\section{On ratio-sharp EKR-sets of $s$-spaces}\label{S:EKR_s-spaces}
Recall that we only consider polar spaces $\PS(n,e,q)$ with $e\ge 1$. First, we recollect the ratio bound for $\Gamma_s$.

\begin{lemma} \label{C: ratio bound (s)}
	The ratio bound $h_{n,s}$ for the independence number of $\Gamma_s$ is 
	$$h_{n,s}= \Gauss{n}{s} \prod\limits_{i=2}^s (q^{n+e-i}+1) =\Gauss{n}{s} \prod\limits_{i=1}^{s-1} (q^{n+e-1-i}+1)$$
 and $\deg(h_{n,s})=s(2n+e-1/2-3s/2)-n-e+1$.
\end{lemma}
\begin{proof}
    By the ratio bound of \Cref{thm:upperbound}, we get 
    \begin{align*}
        h_{n,s}=\frac{|\cF_s|}{q^{n+e-1}+1}.
    \end{align*}
    The number of $s$-spaces in $\PS(n,e,q)$ on the other hand is, according to \Cref{L: extend subspaces}, exactly $|\cF_s|=\Gauss{n}{s}\prod\limits_{i=1}^s(q^{n+e-i}+1)$. The statement follows immediately.
\end{proof}

First, we recall
\begin{align*}
    \lambda_s=  -q^{s(2n+e-1/2-3s/2)-n-e+1} \text{\ and  } \deg(-\lambda_s)=\deg(h_{n,s}).
\end{align*}
It was defined in the introduction that we will call an EKR-set of $s$-spaces \textit{ratio-sharp} if it attains this bound.
%Observe that the ratio bound $h_{n,s}$ is the same as the value $B_s$ that was calculated at the end of \Cref{Section: Classification of ekr}.
In \Cref{Subsection: Maximal examples} we already showed that the ratio bound is tight for $s\in \{1,n \}$.\\

Our goal in this section is to show that the ratio bound is not tight for $1<s<n$. For $s=2$ this is already known \cite{metsch_pairwise-non-opp-lines} and in fact in this case the ratio bound is off by a factor of $q$. 
For this reason, we assume for the remainder of this section that $2 < s < n$. We note that in fact our proof requires these assumptions, see \Cref{cor:EKRboundinperp} and \Cref{P: finish proof}.

%One argument we use (\Cref{cor:EKRboundinperp}) looks at the $s$-spaces in $P^\perp$ that are not incident with $P$, where $P$ is a point. Observe that this does not work for $s=n$. Another argument we use requires $s-1>1$ (see \Cref{P: finish proof}). Therefore, we assume for the remainder of this section that $2<s<n$. 
Our idea is to use the subspace-based antidesign of \Cref{L:m-spaces_m<s} and combine the information provided by this antidesign with geometric observations.  

\begin{notation}
	Let $\EKR_s$ be an EKR-set of $s$-spaces and let $1\leq m <s$ be an integer. For every $m$-space $M$ of $\PS(n,e,q)$ the \emph{weight} of $M$ (with respect to $\EKR_s$) is the number of $s$-spaces of $\EKR_s$ that contain $M$. We call $M$ \emph{heavy} (with respect to $\EKR_s$) if all $s$-spaces that contain $M$ belong to $\EKR_s$; otherwise, we call $M$ \emph{light}.
\end{notation}

We will now prove a lemma that is an analogue of \Cref{L: key for the proof}.

\begin{lemma} \label{L:XYZform<s}
Let $1\leq m<s$ be an integer and $M$ an $m$-space in $\PS(n,e,q)$. Furthermore, let $\EKR_s$ be a ratio-sharp EKR-set of $s$-spaces. Consider the $s$-spaces $S \in \EKR_s$, let $X$ be the number of these with $M \subseteq S$, let $Y$ be the number of these with $M\cap  S^\perp=\{0\}$, and let $Z$ be the number of these with $M\not\subseteq S$ and $M\cap  S^\perp\neq \{0\}$.
\begin{enumerate}[(a)]
\item We have 
\begin{align*}
    & Y= -\lambda_s \cdot q^{-\deg(\Phi_m^s)} \cdot (\Phi_m^s-X), \\
    & Z= X(-\lambda_s \cdot q^{-\deg(\Phi_m^s)}-1)+h_{n,s}+\lambda_s \cdot q^{-\deg(\Phi_m^s)} \cdot \Phi_m^s .
\end{align*}
\item If $M$ is heavy, then $X=\Phi_m^s$. In this case all $s$-spaces $S\in \EKR_s$ satisfy $M\cap S^\perp\neq \{0\}$.
\item If $M$ is light, then 
\begin{align*}
    X\le \Phi_m^s-q^{\deg(\Phi_m^s)} \text{\ and } Y\ge -\lambda_s .
\end{align*}
In this case we have $Y=O(-\lambda_s)=O(h_{n,s})$.
\end{enumerate}
\end{lemma}

\begin{proof} 
For the proof, we use $\Phi:=\Phi_m^s$ and $\lambda:=\lambda_s$ to simplify notation.
From the definition of $v_M$ by \Cref{L:m-spaces_m<s} we see that 
$\charvec_{\EKR_s}^\top v_M=-\lambda X+q^{\deg(\Phi)}Y $ 
so \Cref{C: Intersection m-spaces m<s} implies
\begin{align*}
	\begin{split}
		Y&=q^{-\deg(\Phi)}(-\lambda \Phi +\lambda X)\\
            &=-\lambda \cdot q^{-\deg(\Phi)}(\Phi-X).
		\end{split}
\end{align*}
On the other hand we have $X+Y+Z=h_{n,s}$ and hence
\begin{align*}
    \begin{split}
        Z&= h_{n,s}-X-Y\\
        &= h_{n,s}-X+\lambda \cdot q^{-\deg(\Phi)}(\Phi-X)\\
        &=X(-\lambda \cdot q^{-\deg(\Phi)}-1)+ h_{n,s} +  \lambda \cdot q^{-\deg(\Phi)} \cdot \Phi.
    \end{split}
\end{align*}
If $M$ is heavy then all the $s$-spaces that contain $M$ are in $\EKR_s$, and so $X = \Phi$,
which implies $Y = 0$.\\

 If $M$ is light, then $Y$ is positive and hence there exists $S \in \EKR_s$ such that $M\cap  S^\perp=\{0\}$. The number of $s$-spaces through $M$ that do not meet $S^\perp$ is, according to \Cref{L: heavy power m-space}, exactly $q^{(s-m)(2n-s-m+e-\frac{s-m+1}{2})}=q^{\deg(\Phi)}$. Therefore in this case $X \leq \Phi-q^{\deg(\Phi)}$ and hence 
 \begin{align*}
     Y \geq   -\lambda \cdot q^{-\deg(\Phi)}(\Phi-(\Phi-q^{\deg(\Phi)})) =-\lambda.
 \end{align*}
The last statement follows immediately from the observation $-\lambda \leq Y \leq h_{n,s}$.
\end{proof}

\begin{lemma}
Let $P$ be a point, and let $S$ and $T$ be $s$-subspaces of $\PS(n,e,q)$ with $P\notin S,T \subseteq P^\perp$. Then $S$ and $T$ are opposite in $\PS(n,e,q)$ if and only if $\langle S,P\rangle /P$ and $\langle T,P\rangle /P$ are opposite in the polar space induced in $P^\perp/P$.
\end{lemma}

\begin{proof}
    As $\langle P,T \rangle\subseteq P^\perp$, we have $\langle P,S\rangle^\perp \cap \langle P,T \rangle=P^\perp \cap S^\perp \cap \langle P,T \rangle$. It follows that $S^\perp \cap T=\{ 0\}$ if and only if $\langle P,S \rangle^\perp \cap \langle P,T \rangle=P$.
\end{proof}

In other words, the opposition graph on the $s$-spaces in $P^\perp$ (but not through $P$), is a $q^s$-coclique blow-up of the opposition graph on $s$-spaces in $\PS(n-1,e,q)$. More precisely, this means that every vertex is replaced by a coclique of size $N:=q^s$ and every edge is replaced by a copy of $K_{N,N}$.  The following is an immediate corollary.

\begin{cor}\label{cor:EKRboundinperp}
    For a point $P \in \PS(n,e,q)$, let $\widetilde{\EKR}_s$ be an EKR-set of $s$-spaces in $P^\perp$, none of which are incident with $P$. Then 
    \[|\widetilde{\EKR}_s| \leq q^s \cdot h_{n-1,s}.\]
\end{cor}
%\philipp{I think this corollary is also false. For $e=0$ and $n$ even the Hoffman bound is tight, but it is not tight for $e=0$ and $n$ odd. But if we take $e=0$, $n$ even and $s=n-1$, then we use the Hoffman bound in $\PS(n-1,0,q)$ but there the greek generators are the largest EKR sets and this beats the Hoffman bound. So the case $e=0$ is troublesome...}

\begin{lemma} \label{L: points are very light}
    Let $P$ be a point, and let $\EKR_s$ be a ratio-sharp EKR-set of $s$-spaces. Then $P$ has weight at most $O(q^{\deg(h_{n,s})-n+1})$.
\end{lemma}
\begin{proof}
First, note that $\Phi_1^s \frac{q^n-1}{q^s-1}=h_{n,s}$ and  $-\lambda_s \cdot q^{-\deg(\Phi_1^s)} = q^{n-s}$. 
    By \Cref{L:XYZform<s} we have 
    \begin{align*}
        Z&= X(-\lambda_s \cdot q^{-\deg(\Phi_1^s)}-1)+h_{n,s}+\lambda_s \cdot q^{-\deg(\Phi_1^s)} \cdot \Phi_1^s\\
        &= X(q^{n-s}-1)+h_{n,s}-q^{n-s}\Phi_1^s\\
        &=X(q^{n-s}-1)+h_{n,s}-q^{n-s} \frac{q^s-1}{q^n-1}h_{n,s}\\
        &= X(q^{n-s}-1)+h_{n,s}\left(\frac{q^{n-s}-1}{q^n-1}\right).
    \end{align*}
We know by \Cref{cor:EKRboundinperp} that 
$\deg(Z) \le s+ \deg(h_{n-1,s})=\deg(h_{n,s})-s+1$, so this implies $\deg(X) \le (\deg(h_{n,s})-s+1) -(n-s)=\deg(h_{n,s})-n+1$.
\end{proof}

\begin{lemma} \label{C:everything_is_light}
Let $\EKR_s$ be a ratio-sharp EKR-set of $s$-spaces. Then there are no heavy $m$-spaces for $1\leq m<s$ and $q$ large enough in terms of $n$.
\end{lemma}
\begin{proof}
Every heavy $m$-space is contained inside a heavy $(s-1)$-space, so it suffices to show that there are no heavy $(s-1)$-spaces. Assume that $M$ is a heavy $(s-1)$-space. Since $M$ is heavy, we apply \Cref{L:XYZform<s} (b) and get that every element $S$ in $\EKR_s$ satisfies $M \cap S^\perp \neq \{0\}$. In other words, as we range over the points $P$ of $M$, we see every element of $\EKR_s$ in the space $P^\perp$ for at least one point $P$. By \Cref{cor:EKRboundinperp}, the number of $s$-spaces skew to $M$ is therefore at most
    \begin{align*}
        \Gauss{s-1}{1} \cdot q^s \cdot h_{n-1,s}.
    \end{align*}
   Since $\deg(h_{n-1,s})+2s-1=\deg(h_{n,s})$, we have that the number of elements in $\EKR_s$ skew to $M$ is $ O(h_{n,s} \cdot q^{-1})$. This implies that $\Theta(h_{n,s})$ many $s$-spaces of $\EKR_s$ meet $M$.\\

   On the other hand, the number of $s$-spaces that meet $M$ in at least a point is, according to the previous lemma, at most $\Gauss{s-1}{1}\cdot O(q^{\deg(h_{n,s})-n+1})=O(q^{\deg(h_{n,s})-(n-s)-1})$. This number is $o(h_{n,s})$ as $n>s$. This is a contradiction.
\end{proof}

Now we know that for large enough $q$ there are no heavy subspaces to consider.
Next, we prove two lemmata that in conjunction provide a ladder argument on how one can make the weight of a light $m$-space arbitrarily small.  For the remaining part of this paper, we define $h:=\deg(h_{n,s})$ to simplify notation.

\begin{lemma} \label{L:down_the_waterfall}
Let $\EKR_s$ be a ratio-sharp EKR-set of $s$-spaces.  Furthermore, let $1 \leq m \leq s-2$ and let $M$ be a light $m$-space. Suppose that all $(m+1)$-spaces are light and have weight $O(q^a)$, where $a$ is a real number. Then $M$  has weight $O(q^{2n-s-m+e-2+a})$.       
\end{lemma}

\begin{proof}
    Since $M$ is light, \Cref{L:XYZform<s} (c) implies that there are at least $q^h$ many $s$-spaces $S$ such that $M \cap S^\perp= \{0\}$, which implies that $M^\perp$ meets $S$ in an $(s-m)$-space. Let $S$ be such a space.  
    Then $M^\perp \cap S^\perp $ is a degenerate polar space of rank $n-m$ with radical $M^\perp \cap S$ of rank $s-m$. The number of points in the radical is therefore $\Gauss{s-m}{1}$ and the number of points in a polar space of rank $n-s$ is $ \Gauss{n-s}{1} (q^{n-s+e-1}+1)$. Therefore, the number of points in $M^\perp \cap S^\perp$ is $O(q^{s-m-1}\cdot q^{n-s-1+n-s+e-1} \cdot q   )=O(q^{2n-s-m+e-2})$.
    
    %Then $M^\perp \cap S^\perp $ is a cone over a rank $n-s$ polar space whose apex is a subspace of rank $s-m$ \red{maybe better to reuse terminology from before:  ... is a degenerate polar space of rank $n-m$ with radical $M^\perp \cap S$ of rank $s-m$. Also in the following apex -> radical}. 
    %The number of points in the apex of the cone is therefore $\Gauss{s-m}{1}$ and the number of points in a polar space of rank $n-s$ is $ \Gauss{n-s}{1} (q^{n-s+e-1}+1)$. Therefore, the number of points in $M^\perp \cap S^\perp$ is $O(q^{s-m-1}\cdot q^{n-s-1+n-s+e-1} \cdot q   )=O(q^{2n-s-m+e-2})$.

    Every $s$-space in $\EKR_s$ containing $M$ is contained in $M^\perp$ and must also contain a point of $S^\perp$. Hence, it contains a point in $M^\perp \cap S^\perp$. In other words, it must contain an $(m+1)$-space $\langle M,X \rangle$, with $X \in M^\perp \cap S^\perp$. Since there are $O(q^{2n-s-m+e-2})$
    choices for $X$ and $O(q^a)$ $s$-spaces through every $(m+1)$-space, we conclude that a light $m$-space has weight $O(q^{2n-s-m+e-2+a})$. 
\end{proof}

Once we reach the bottom of the ladder, we can go back up to the top with the following lemma.

\begin{lemma} \label{L:up_the_waterfall}
Let $\EKR_s$ be a ratio-sharp EKR-set of $s$-spaces.
   Furthermore, suppose that all points have weight $O(q^b)$, where $b$ is a real number. Then every light $(s-1)$-space has weight at most $O(q^{2n-s+e-h-2+b})$.
\end{lemma}

\begin{proof}
    Let $M$ be a light $(s-1)$-space, then by \Cref{L:XYZform<s} (c), there are at least $q^h$ elements $S \in \EKR_s$ such that $M^\perp \cap S$ is a point not in $M$. 
    Since every point has weight $O(q^b)$, this implies that there are at least $\Omega(q^{h-b})$ points with non-zero weight in $M^\perp$. 
    Since $q^{s-1}$ points in $M^\perp$ can project to the same point in $M^\perp/M$, we can find a set $X$ of $\Omega(q^{h-b-s+1})$ points in the quotient space, which is a non-degenerate polar space of rank $n-s+1$. According to \Cref{L: gaussian coefficient} the number of points in this polar space is $\Phi_0^1(n-s+1,e,q)$ which is $\Theta(q^{2n-2s+e})$. 
    
    Now every $s$-space in $\EKR_s$ through $M$ corresponds to a point in $M^\perp / M$, and this point needs to be collinear to all points in $X$. %If there are $\Omega(q^{h-b-s+1})$ points in $X$, we can find $h-b-s+2$ points in general position in $X$, meaning their span has dimension $\geq h-b-s+2$ in the ambient vector space.     

Any point that is collinear to all points of $X$, must therefore be contained in $X^\perp$. As $M^\perp/M$ has rank $n-s+1$, we have that $M^\perp/M$ contains $\Theta(q^{2n-2s+e})$ points. 
%Define $t:=2n-2s+e$. Our goal is to show that $X^\perp$ contains $O(q^{t-(h-b-s+2)})=O(q^{2n-s-h+e-2+b})$ points, as this implies that the weight of an $(s-1)$-space is at most $O(q^{2n-s+e-h-2+b})$.
According to \Cref{L: points in X perp}, we have that $\langle X \rangle^\perp$ contains at most $O(q^{(2n-2s+e)-(h-b-s+2)})=O(q^{2n-s-h+e-2+b})$ points.
In other words, the weight of an $(s-1)$-space is at most $O(q^{2n-s+e-h-2+b})$. 
\end{proof}

\begin{prop} \label{P: finish proof}
   For $2<s<n$ and $q$ large enough (in terms of $n$) there is no ratio-sharp EKR-set of $s$-spaces.
\end{prop}

\begin{proof}
Assume that $\EKR_s$ is a ratio-sharp EKR-set of $s$-spaces. We show that every point has weight $O(1)$. This implies that there does not exist a ratio-sharp EKR-set of $s$-spaces for $q$ large enough, the number of $s$-spaces in $\EKR_s$ is at most $O(\Phi_0^1)$ and $\deg(\Phi_0^1)<h$. 
%\philipp{I think this is not always true (take $e=0$ and $s=n-1$ and $n=4$.)}

\Cref{C:everything_is_light} states that there are no heavy $m$-spaces for $1\leq m<s$  in $\PS(n,e,q)$, so there are only light $m$-spaces. We show that using \Cref{L:up_the_waterfall} and \Cref{L:down_the_waterfall} inductively, we can decrease the weight of a point. Note that we use $s>2$, otherwise there would be no step to take from $(s-1)$-spaces to points. 
\Cref{L: points are very light} implies that as an induction base, we can assume that every point has weight $O(q^{b_0})$, with $b_0=h-n+1$. Our induction step works as follows.
%We remark that our starting weight for the induction is given by \Cref{L: points are very light}.
 
\textbf{Step 0.} A point has weight $O(q^{b})$. 

\textbf{Step 1.} By \Cref{L:up_the_waterfall} we find that all $(s-1)$-spaces have weight $O(q^{2n-s+e-h+b-2})$. 

\textbf{Step 2.} By \Cref{L:down_the_waterfall} we find that all $(s-2)$-spaces have weight $$O(q^{2n-s-(s-2)+e-2} \cdot q^{2n-s+e-h+b-2}).$$

$\dots$ \\

\textbf{Step $s-1$.} By \Cref{L:down_the_waterfall} we find that all points have weight $O(q^x)$, where 
\[x = \sum_{m=1}^{s-2}(2n-s+e-2-m)+2n-s+e-h+b-2.\]
Now we will prove that $x < b$, which means that the weight of a point decreased by an order of magnitude $q$.
This implies that we can use the above algorithm repeatedly to eventually find that points have weight $O(1)$.
Recalling $h = s(2n+e-1/2-3s/2)-n-e+1$, proving $x<b$ is equivalent to 
\begin{align*}
    & -h + (s - 2) (2n-s+e-2)- \frac{(s-2)(s-1)}{2} + 2 n - s + b + e - 2< b \\
    &\Leftrightarrow -h + h-n+s +b<b\\
    & \Leftrightarrow s<n,
\end{align*}
which is true, since we assumed $2<s<n$ for this section.
\end{proof}

Recall that it was shown already in \cite{metsch_pairwise-non-opp-lines} that the ratio bound for $s=2$ is off by a factor of $q$. Therefore, the previous proposition together with the observations at the end of \Cref{Section: Classification of ekr} prove \Cref{thm:main}.\\

\section{Concluding remarks}

Although we solved the classification problem for ratio-sharp EKR-sets of chambers in polar spaces with $e\geq 1$, if $q$ is large enough compared to $n$, several natural questions remain open. First, it would be very interesting, but probably quite difficult, to crack the case $e = 1/2$ and $n$ odd.  
%We remark that one can adapt the proof of this paper to obtain classification for ratio-sharp EKR-sets of flags with $n\notin J$ in the case $e = 1/2$ and $n$ odd. 
%\SM{Really? Didnt' remember we could haha.} \philipp{Actually we cannot, right? because we run again into the $n-1$ is even trap?}
%\SM{So we should remove this?}
The problem seems to boil down to a lack of understanding of maximum EKR-sets of generators and solving this is the first step towards classification.

\begin{problem}
    Show that for $n \geq 5$ odd, a maximum EKR-set of generators in $\PS(n,1/2,q)$ consists of all generators through a point.
\end{problem}

Note that for $n = 3$, taking all generators through a point does not produce the largest EKR-set \cite{EKR_polarspace_generators}.

In our second main result, we showed that the ratio bound is not sharp for flags of type $J$ with $1,n\notin J$. A natural problem is hence to find the correct, or at least asymptotically sharp, upper bounds in this case. The only known result in this direction comes from \cite{metsch_pairwise-non-opp-lines} and concerns $J=\{2\}$. 

\begin{problem}
    Find good (asymptotic) upper bounds for EKR-sets of $s$-spaces in $\PS(n,e,q)$, $2<s<n$.
\end{problem}

A first step would be to show that the ratio bound is off by at least some power of $q$ (up to multiplicative constants), which is indeed the case for $s = 2$ \cite{metsch_pairwise-non-opp-lines}. This question is closely related to the variation of the center conjecture for spherical buildings mentioned in the introduction: for which size of maximal EKR-sets of flags of type $J$ do we find a center? Our results prove this in some cases where the ratio bound is attained, and \cite{metsch_pairwise-non-opp-lines} confirms it for maximum EKR-sets of lines. Note that in the latter case, unless $n$ is small, the center is really a flag and not a single subspace. It is not entirely unreasonable to expect that this phenomenon continues to hold for large maximal EKR-sets of flags in type $B$.

\begin{problem}
    Do large maximal EKR-sets of flags in spherical buildings of type $B$ always have a center? What is `large' in this context?
\end{problem}

Note that the answer to this question in full generality is `no', since maximum EKR-sets of generators in the case $e = 0$ and $n$ odd have no center, see \Cref{E: EKR for s-spaces} (c); the same can be said for the case $e=1$ and $n$ odd, see \Cref{E: EKR for s-spaces} (d). %Also (d) is another counter example
Besides these cases, we believe the answer to the first part of this problem to be affirmative. \\

Finally, the antidesign method used in this work and \cite{klaus_jesse_philipp} appears to be quite general and powerful. It would be very interesting to see in which other settings this method works.

\section*{Acknowledgements}

The authors thank Jesse Lansdown for initial discussions at an early stage of this work.
The second author would like to acknowledge the support of the Vrije Universiteit Brussel 
(grant OZR3637: ``Discrete Structures, and their applications in Data Science''), where part of this work was done while the second author was a visitor. The third author is supported by a postdoctoral fellowship 1267923N from the Research Foundation Flanders (FWO).

%\bibliographystyle{plainurl}
%\bibliography{bib_polar.bib}

\end{document}